\pdfoutput=1

\documentclass[11pt,a4paper]{amsart}
\usepackage[english]{babel}
\usepackage{amsmath,amssymb,amsthm,latexsym,MnSymbol}
\usepackage[utf8]{inputenc}
\usepackage[T2A,T1]{fontenc}
\usepackage{url}
\usepackage[pdftex]{graphicx}
\usepackage{tikz}
\usepackage{tikz-cd}
\usepackage{hyperref}
\usepackage[a4paper, portrait]{geometry}
\usepackage{setspace}

\newcommand{\Z}{\mathbf{Z}}
\newcommand{\Zhat}{\hat{\Z}}
\newcommand{\Q}{\mathbf{Q}}
\newcommand{\Qb}{\overline{\Q}}

\newcommand{\R}{\mathbf{R}}
\newcommand{\CC}{\mathbf{C}}

\newcommand{\A}{\mathbf{A}}

\newcommand{\TT}{\mathbf{T}}

\newcommand{\B}{\mathrm{B}}
\newcommand{\dR}{\mathrm{dR}}
\newcommand{\et}{\textrm{ét}}
\newcommand{\OK}{\mathcal{O}_K}

\DeclareSymbolFont{cyrillic}{T2A}{cmr}{m}{n}
\DeclareMathSymbol{\Sha}{\mathalpha}{cyrillic}{216}
\DeclareMathSymbol{\BB}{\mathalpha}{cyrillic}{193}

\DeclareMathOperator{\Aut}{Aut}
\DeclareMathOperator{\CH}{CH}
\DeclareMathOperator{\CHM}{CHM}

\DeclareMathOperator{\Cot}{Cot}
\DeclareMathOperator{\Fil}{Fil}
\DeclareMathOperator{\End}{End}
\DeclareMathOperator{\Gal}{Gal}
\DeclareMathOperator{\GL}{GL}
\DeclareMathOperator{\Hom}{Hom}
\DeclareMathOperator{\id}{id}
\DeclareMathOperator{\Ind}{Ind}
\DeclareMathOperator{\Lie}{Lie}
\DeclareMathOperator{\nr}{nr}
\DeclareMathOperator{\op}{op}
\DeclareMathOperator{\ord}{ord}

\DeclareMathOperator{\Res}{Res}

\DeclareMathOperator{\rr}{rr}
\DeclareMathOperator{\SL}{SL}
\DeclareMathOperator{\Spec}{Spec}

\newtheorem{thm}{Theorem}
\newtheorem*{thm*}{Theorem}

\newtheorem{lem}[thm]{Lemma}
\newtheorem{pro}[thm]{Proposition}
\newtheorem{cor}[thm]{Corollary}
\newtheorem*{cor*}{Corollary}
\newtheorem{conjecture}[thm]{Conjecture}
\theoremstyle{definition}

\newtheorem*{notations}{Notations}
\theoremstyle{remark}
\newtheorem{remark}{Remark}

\newtheorem*{remarks*}{Remarks}
\newtheorem{example}{Example}

\begin{document}


\title[Non-critical equivariant $L$-values]{Non-critical equivariant $L$-values\\ of modular abelian varieties}

\author[F. Brunault]{François Brunault}

\date{\today}

\address{ÉNS Lyon, Unité de mathématiques pures et appliquées, 46 allée d'Italie, 69007 Lyon, France}

\email{francois.brunault@ens-lyon.fr}
\urladdr{http://perso.ens-lyon.fr/francois.brunault}


\begin{abstract}
We prove an equivariant version of Beilinson's conjecture on non-critical $L$-values of strongly modular abelian varieties over number fields. The proof builds on Beilinson's theorem on modular curves as well as a modularity result for endomorphism algebras. As an application, we prove a weak version of Zagier's conjecture on $L(E,2)$ and Deninger's conjecture on $L(E,3)$ for non-CM strongly modular $\Q$-curves.
\end{abstract}

\maketitle

The purpose of this article is to use the full strength of Beilinson's theorem on modular curves to prove the following result.

\begin{thm}\label{main thm}
Let $A$ be an abelian variety defined over a number field $K$ such that the Hasse-Weil $L$-function $L(A/K,s)$ is a product of $L$-functions of newforms of weight $2$ without complex multiplication on congruence subgroups of $\SL_2(\Z)$. Then for every integer $n \geqslant 2$, the weak form of Beilinson's conjecture on $L(A/K,n)$ holds.
\end{thm}

We in fact prove a slightly stronger result, namely an equivariant version of Beilinson's conjecture for the Chow motive $H^1(A/K)$ with coefficients in the endomorphism algebra of $A$, at every non-critical integer (see Corollary \ref{cor 1}).

In the case $A$ is an elliptic curve defined over $\Q$, Theorem \ref{main thm} is a direct consequence of Beilinson's theorem on the modular curve $X_1(N)$, together with the existence of a modular parametrization $X_1(N) \to A$. In the case $K=\Q$, the simple abelian varieties satisfying the assumption of Theorem \ref{main thm} are precisely the classical modular abelian varieties $A_f$ attached to newforms $f$ of weight $2$ on $\Gamma_1(N)$. Note that in general $A_f$ may split over an abelian extension of $\Q$, and Theorem \ref{main thm} applies to every factor of $A_f$.

In the particular case of elliptic curves, Theorem \ref{main thm} has the following consequence on Zagier's conjecture (see \cite{wildeshaus:ezc} for the statement of Zagier's conjecture). Recall that a $\Q$-curve is an elliptic curve $E$ over a number field $K$ which is isogenous to all its Galois conjugates.

\begin{cor}\label{main cor}
Let $E$ be a $\Q$-curve without complex multiplication over a number field $K$ such that $L(E/K,s)$ is a product of $L$-functions of newforms of weight $2$. Then the weak form of Zagier's conjecture on $L(E/K,2)$ holds.
\end{cor}

Note that Corollary \ref{main cor} is a generalization of \cite[Thm 1, Cor]{brunault:LEF}, where we considered the case of the base change of an elliptic curve $E/\Q$ to an abelian number field. Note also that the case of CM elliptic curves was already worked out by Deninger \cite{deninger:hecke1,deninger:hecke2}.

Thanks to the work of Ribet and the proof of Serre's conjecture due to Khare and Wintenberger, $\Q$-curves are known to be modular in the sense that they admit a modular parametrization by $X_1(N)_{\Qb}$ for some integer $N$. If this parametrization happens to be defined over an abelian number field, then Corollary \ref{main cor} applies. For instance, Corollary \ref{main cor} applies to every non-CM $\Q$-curve $E$ defined over a quadratic field $K$ whose isogeny $E \to E^{\sigma}$ is also defined over $K$. To our knowledge, this is the first proof of Zagier's conjecture on $L(E,2)$ for a non-CM elliptic curve which is genuinely defined over a number field. 

Using Goncharov's results \cite{goncharov:LE3}, we also get the following consequence on Deninger's conjecture on $L(E,3)$ (see \cite{deninger:higher,goncharov:LE3} for an account of Deninger's conjecture).

\begin{cor}\label{main cor 2}
Let $E$ be a $\Q$-curve without complex multiplication over a number field $K$ such that $L(E/K,s)$ is a product of $L$-functions of newforms of weight $2$. Then the weak form of Deninger's conjecture on $L(E/K,3)$ holds.
\end{cor}

The proof of Theorem \ref{main thm} builds on the profound results of Beilinson \cite{beilinson:2} on values of $L$-functions associated to modular forms. The main technical ingredients are a Hecke-equivariant version of Beilinson's theorem, together with a modularity result for endomorphism algebras. More precisely, we show that every endomorphism of a modular abelian variety $A_f$ which is defined over an abelian extension of $\Q$ is of automorphic origin, making slightly more precise a theorem of Ribet (see \S \ref{modular endo}).

In order to deal with Beilinson's conjecture for the factors of $J_1(N)$ over $\Qb$, we are naturally led to use the equivariant formalism developed by Burns and Flach \cite{burns-flach}. The need for this formalism can be explained as follows. Imagine that an abelian variety $A/\Q$ decomposes as a product $A_1 \times A_2$, so that the $L$-function of $A$ decomposes as $L(A,s) = L(A_1,s) L(A_2,s)$. If Beilinson's conjecture (at some integer) holds for $A_1$ and $A_2$, then it holds for $A$, but the converse is not clear. If we use the equivariant $L$-function instead, then Beilinson's conjecture for $A$ implies readily the conjecture for $A_1$ and $A_2$.

Throughout the article, we will work with Chow motives of abelian varieties. Although this could be avoided, we chose this point of view because we need to transport Beilinson's theorem from a modular curve to its Jacobian variety, and this is most naturally done by comparing the Chow motive of this curve with the Chow motive of its Jacobian.

This work originates in an invitation at Kyoto University in October 2010. I gave a lecture on the results of \cite{brunault:LEF} and Prof. Hida suggested that the same method could work for $\Q$-curves. I would like to thank Prof. Hida for his valuable suggestion. I am also grateful to the referee for several helpful comments. I would like to thank Xavier Guitart for providing me a reference allowing me to drop the assumption that $K$ is Galois in Theorem \ref{main thm}. Finally, I would like to thank Frédéric Déglise, Gabriel Dospinescu, Vincent Pilloni for stimulating discussions on related topics.

\section{The equivariant Beilinson conjecture}

In this section we review the formulation of the equivariant Beilinson conjecture for Chow motives endowed with an action of a semisimple algebra. We consider Chow motives with coefficients in a number field $E$. The results in this section remain valid if $E$ is an arbitrary subfield of $\Qb$ (although in this case $E \otimes_\Q \R$ need not be semisimple), but we won't need this level of generality.

\begin{notations}
For any $\Q$-vector space $V$ and any field $F$ of characteristic $0$, we put $V_F=V \otimes_\Q F$. For any ring $R$, we denote by $Z(R)$ the center of $R$.
\end{notations}

\subsection{Chow motives with coefficients}

Let us review some background material on Chow motives.

Let $K$ and $E$ be two number fields. The category $\CHM_K(E)$ of Chow motives defined over $K$ with coefficients in $E$ consists of triples $(X_d,p,n)$ where $X_d$ denotes a $d$-dimensional smooth projective $K$-variety, $p \in \CH^d(X \times_K X) \otimes E$ is an idempotent and $n \in \Z$ (\cite[Chapter 2]{murre-nagel-peters}, \cite[\S 4]{jannsen:deligne}). Morphisms between two objects $M=(X_d,p,m)$ and $N=(Y_e,q,n)$ in $\CHM_K(E)$ are given by
\begin{equation*}
\Hom(M,N) = q \circ \bigl(\CH^{d+n-m}(X \times_K Y) \otimes E)\bigr) \circ p.
\end{equation*}
Note that $\End(M)$ is an $E$-algebra, and every idempotent $e \in \End(M)$ has kernel and image in $\CHM_K(E)$. Note that we use contravariant notations for our motives: there is a contravariant functor sending a smooth projective $K$-variety $X$ to the motive $h(X)=(X,\Delta_X,0)$, where $\Delta_X$ denotes the class of the diagonal in $X \times X$. Any morphism $\phi : X \to Y$ gives rise to a morphism $\phi^* : h(Y) \to h(X)$, defined by the class of the graph of $\phi$.

Let $M=(X_d,p,n) \in \CHM_K(E)$ be a Chow motive, and let $0 \leqslant i \leqslant 2d$ be an integer. We will denote by the formal notation $H^i(M)$ the following system of realizations:

\begin{itemize}
\item for any embedding $\sigma : K \hookrightarrow \CC$, the Betti realization
\begin{equation*}
H^i_{\B,\sigma}(M) = p^* H^i_{\B}(X_\sigma(\CC),E(n));
\end{equation*}
\item the de Rham realization
\begin{equation*}
H^i_{\dR}(M) = p^* (H^i_{\dR}(X) \otimes_\Q E);
\end{equation*}
\item for any prime $\ell$, the $\ell$-adic étale realization
\begin{equation*}
H^i_{\et}(M) = p^* \bigl[ H^i_{\et}(X_{\overline{K}},\Z_\ell(n)) \otimes_{\Z_\ell} (E \otimes_\Q \Q_\ell)\bigr].
\end{equation*}
\end{itemize}
These realizations are linked by the following comparison theorems. For any embedding $\sigma : K \hookrightarrow \CC$, we have an isomorphism of $E \otimes \CC$-modules (Grothendieck's theorem)
\begin{equation}\label{I sigma}
I_{\sigma} : H^i_{\B,\sigma}(M) \otimes_\Q \CC \xrightarrow{\cong} H^i_{\dR}(M) \otimes_{K,\sigma} \CC,
\end{equation}
and for any embedding $\tilde{\sigma} : \overline{K} \hookrightarrow \CC$ extending $\sigma$, we have an isomorphism of $E \otimes \Q_\ell$-modules
\begin{equation*}
I_{\ell,\tilde{\sigma}} : H^i_{\B,\sigma}(M) \otimes_\Q \Q_\ell \xrightarrow{\cong} H^i_{\et}(M).
\end{equation*}
By definition, the \emph{weight} of $H^i(M)$ is $i-2n$. Recall the Hodge decomposition $H^i_{\B,\sigma}(M) \otimes_{\Q} \CC = \bigoplus_{a+b=i-2n} H^{a,b}_\sigma(M)$. The $E$-vector space $H^i_{\dR}(M)$ carries a decreasing filtration $(\Fil^k H^i_{\dR}(M))_{k \in \Z}$ such that
\begin{equation*}
(\Fil^k H^i_{\dR}(M)) \otimes_{K,\sigma} \CC = I_\sigma \Bigl(\bigoplus_{a \geq k} H^{a,b}_\sigma(M)\Bigr).
\end{equation*}
We put $H^i_\B(M) = \bigoplus_{\sigma : K \hookrightarrow \CC} H^i_{\B,\sigma}(M)$, so that the various isomorphisms (\ref{I sigma}) combine to give
\begin{equation*}
I_\infty : H^i_B(M) \otimes_\Q \CC \xrightarrow{\cong} H^i_\dR(M) \otimes_\Q \CC.
\end{equation*}
We put $\Fil^k (H^i_B(M) \otimes_\Q \CC) = I_\infty^{-1}(\Fil^k H^i_{\dR}(M) \otimes_\Q \CC)$ and $\overline{\Fil}^k = (\id \otimes c)(\Fil^k)$, where $c$ denotes the complex conjugation. Note that $\overline{\Fil}^k (H^i_B(M) \otimes_\Q \CC) = \bigoplus_{\sigma} \bigoplus_{b \geq k} H^{a,b}_{\sigma}(M)$.

Let $A$ be an $E$-algebra. We denote by $\CHM_K(A)$ the category of Chow motives in $\CHM_K(E)$ endowed with an action of $A$. Its objects are pairs $(M,\rho)$ with $M \in \CHM_K(E)$ and $\rho : A \to \End(M)$ is a morphism of $E$-algebras. Morphisms in $\CHM_K(A)$ are morphisms in $\CHM_K(E)$ commuting with the action of $A$. The category $\CHM_K(A)$ is additive but not abelian. If $M \in \CHM_K(A)$ then all realizations of $M$ have natural structures of left $A$-modules, and the comparison isomorphisms are $A$-linear. If $e \in A$ is an idempotent then we may define $e(M) \in \CHM_K(eAe)$.

Assume $A$ is finite-dimensional and semisimple. Conjecturally, we then have an equivariant $L$-function $L({}_A H^i(M),s)$ taking values in the center $Z(A_\CC)$ of $A_\CC := A \otimes_{\Q} \CC$ \cite[\S 4]{burns-flach}. This function should be meromorphic in the sense that for every embedding $\sigma$ of $E$ into $\CC$, the function $s \mapsto L({}_A H^i(M),s)^\sigma \in Z(A \otimes_{E,\sigma} \CC)$ is meromorphic. In the case $A=E=\Q$, we recover the usual complex-valued $L$-function $L(H^i(M),s)$. For any $M \in \CHM_K(A)$ and any integer $n \in \Z$, we denote by $M(n) := M \otimes E(n)$ the $n$-th Tate twist of $M$. Recall that
\begin{equation}\label{L shift}
L({}_A H^i(M(n)),s) = L({}_A H^i(M),s+n) \qquad (s \in \CC).
\end{equation}

For any motive $M=(X_d,p,n)$, we denote by $M^* = (X_d,{}^t p,d-n)$ the dual motive. If $A$ acts on $M$ then $A^{\op}$ acts on $M^*$. Conjecturally, the function $L({}_A H^i(M),s)$ extends to a meromorphic function on $\CC$ and there is a functional equation relating $L({}_A H^i(M),s)$ and $L({}_{A^{\op}} H^{2d-i}(M^*),1-s)$.

\begin{remark}
According to \cite[6.2]{scholl:motives}, given a smooth projective variety $X/K$ of dimension $d$, there should be a direct sum decomposition $h(X) = \bigoplus_{i=0}^{2d} h^i(X)$ in the category $\CHM_K(\Q)$. Such a decomposition is known in the case $X/K$ is an abelian variety \cite{deninger-murre}, which is the only case we will consider in this paper. In this case, we even have canonical Chow-Künneth projectors $p_0,\ldots,p_{2d} \in \End(h(X))$ such that $(X,p_i,0) \cong h^i(X)$ for every $0 \leqslant i \leqslant 2d$. These projectors are compatible with morphisms of abelian varieties \cite[Prop 3.3]{deninger-murre}. Thus for every morphism $\phi : X \to Y$ of abelian varieties over $K$, we get morphisms $\phi^* : h^i(Y) \to h^i(X)$.
\end{remark}


\subsection{Relative \texorpdfstring{$K$}{K}-theory}

Let $E$ be a number field, and let $A$ be a finite-dimensional semisimple $E$-algebra. The $A$-equivariant versions of the Beilinson conjectures are most conveniently formulated using the relative $K$-group $K_0(A,\R)$. Recall that $K_0(A,\R)$ is an abelian group generated by triples $(X,f,Y)$ where $X$ and $Y$ are finitely generated $A$-modules and $f : X_\R \to Y_\R$ is an isomorphism of $A_\R$-modules \cite[p. 215]{swan}. Note that $X$ and $Y$ are automatically projective since $A$ is semisimple. This group sits in an exact sequence \cite[Thm 15.5]{swan}
\begin{equation}\label{K0AR long}
K_1(A) \to K_1(A_\R) \xrightarrow{\delta} K_0(A,\R) \to K_0(A) \to K_0(A_\R).
\end{equation}

\begin{lem}\label{lem K0AR}
The map $K_0(A) \to K_0(A_\R)$ is injective.
\end{lem}

\begin{proof}
It suffices to consider the case where $A$ is simple, in other words $A=M_n(D)$ for some division algebra $D$ over $E$. Since $K_0(M_n(B))$ is canonically isomorphic to $K_0(B)$ for every ring $B$, it suffices to prove the injectivity of $K_0(D) \to K_0(D_\R)$. Let $\sigma$ be an embedding of $E$ into $\CC$. Since $D \otimes_{E,\sigma} \CC$ is semisimple, there exists a ring morphism $D \otimes_{E,\sigma} \CC \to M_n(\CC)$. Now the composite map
\begin{equation*}
\Z \cong K_0(D) \to K_0(D_\R) \to K_0(D \otimes_{E,\sigma} \CC) \to K_0(M_n(\CC)) \cong \Z
\end{equation*}
sends $1$ to $n$, thus is injective.
\end{proof}

Since $A$ is semisimple, we have a reduced norm map $\nr : K_1(A) \to Z(A)^\times$ \cite[\S 45A]{curtis-reiner2}.

\begin{lem}\label{lem nr}
The reduced norm map $\nr$ is injective.
\end{lem}

\begin{proof}
We may assume that $A$ is a central simple algebra over $E$, in which case the result follows from \cite[(45.3)]{curtis-reiner2}.
\end{proof}

The algebras $A_\R$ and $A_\CC$ are semisimple so we also have reduced norm maps on $K_1(A_\R)$ and $K_1(A_\CC)$ making the following diagram commute:
\begin{equation} \label{diag nr}
\begin{tikzcd}
K_1(A) \rar \dar{\nr} & K_1(A_\R) \rar \dar{\nr_\R} & K_1(A_\CC) \dar{\nr_\CC} \\
Z(A)^\times \rar & Z(A_\R)^\times \rar & Z(A_\CC)^\times.
\end{tikzcd}
\end{equation}
By Lemma \ref{lem nr} and diagram (\ref{diag nr}), the map $K_1(A) \to K_1(A_\R)$ is injective. The exact sequence (\ref{K0AR long}) thus simplifies to
\begin{equation}\label{K0AR short}
0 \to K_1(A) \to K_1(A_\R) \xrightarrow{\delta} K_0(A,\R) \to 0.
\end{equation}

\begin{example}
Let us consider the classical case, namely $A=E=\Q$. Then $K_1(A)=\Q^\times$ and $K_1(A_\R)= \R^\times$ so that $K_0(A,\R)$ can be identified with $\R^\times/\Q^\times$. Moreover, using this identification, the class of $(X,f,Y)$ in $K_0(A,\R)$ is none other than the determinant of $f$ with respect to bases of $X$ and $Y$.
\end{example}

\begin{lem}\label{lem K1 cartesien}
The map $\nr_\R$ is injective and the map $\nr_\CC$ is an isomorphism. Moreover, the left-hand square of diagram (\ref{diag nr}) is Cartesian: identifying the groups $K_1(A)$, $K_1(A_\R)$ and $Z(A)^\times$ with subgroups of $Z(A_\R)^\times$, we have
\begin{equation}\label{eq K1 cartesien}
Z(A_\R)^\times = Z(A)^\times \cdot K_1(A_\R) \qquad \textrm{and} \qquad K_1(A) = Z(A)^\times \cap K_1(A_\R).
\end{equation}
\end{lem}

\begin{proof}
We may assume that $A$ is a central simple algebra over $E$. The injectivity of $\nr_\R$ and the bijectivity of $\nr_\CC$ are proved in \cite[(45.3)]{curtis-reiner2}. Let $\Sigma_\infty$ be the set of archimedean places of $E$. For any $v \in \Sigma_\infty$, let $A_v := A \otimes_E E_v$, so that $A_\R \cong \prod_{v \in \Sigma_\infty} A_v$ and $Z(A_\R)^\times = (E \otimes_\Q \R)^\times = \prod_{v \in \Sigma_\infty} E_v^\times$. Let $\Sigma$ be the set of places $v \in \Sigma_\infty$ such that $E_v = \R$ and $A_v$ is isomorphic to a matrix algebra over the real quaternions. By \cite[(45.3)]{curtis-reiner2}, we have
\begin{align}
\label{nr} \nr(K_1(A)) & = \{x \in E^\times ; x_v>0 \textrm{ for every } v \in \Sigma\}\\
\label{nrR} \nr_\R(K_1(A_\R)) & = \{(x_v)_{v \in \Sigma_\infty} ; x_v>0 \textrm{ for every } v \in \Sigma\}.
\end{align}
In particular the image of $\nr_\R$ contains the connected component of identity in $Z(A_\R)^\times$. Since $E^\times$ is dense in $(E \otimes_\Q \R)^\times$, the first identity of (\ref{eq K1 cartesien}) follows. The second equation is an immediate consequence of (\ref{nr}) and (\ref{nrR}).
\end{proof}

Following the terminology of \cite[\S 4.2]{burns-flach}, the \emph{extended boundary map} $\hat{\delta} : Z(A_\R)^\times \to K_0(A,\R)$ is the unique extension of $\delta$ to $Z(A_\R)^\times$ which vanishes on $Z(A)^\times$ (such an extension exists and is unique by Lemma \ref{lem K1 cartesien}).

Let $X,Y,Z$ be finitely generated $A$-modules, together with a short exact sequence of $A_\R$-modules
\begin{equation*}
0 \to X_\R \xrightarrow{\alpha} Y_\R \xrightarrow{\beta} Z_\R \to 0.
\end{equation*}
Since $Z_\R$ is projective over $A_\R$, this sequence splits and the map $\beta$ admits a section $s : Z_\R \to Y_\R$. Then the element $\vartheta=(X \oplus Z,\alpha \oplus s,Y) \in K_0(A,\R)$ is independent of the choice of $s$.

\subsection{Statement of the conjecture in the region of convergence} \label{conj statement}

In this section we state an equivariant version of Beilinson's conjecture on special values of $L$-functions in the region of absolute convergence. This conjecture is a particular case of the equivariant Tamagawa number conjecture of Burns and Flach \cite[\S 4.3]{burns-flach}. Since we don't consider the integrality part of the conjecture in this article, the formulation becomes in fact much simpler.

Let $E$ be a number field, and let $A$ be a finite-dimensional semisimple $E$-algebra. Fix a Chow motive $M = (X_d,p,0,\rho)$ in $\CHM_K(A)$ and an integer $0 \leqslant i \leqslant 2d$. Whenever defined, the equivariant $L$-function $L({}_A H^i(M),s)$ converges absolutely in the region $\Re(s)>\frac{i}{2}+1$, and because of the Euler product, its values at integers in this region belong to $Z(A_\R)^\times$.

Fix an integer $n >\frac{i}{2}+1$. The conjecture on $L({}_A H^i(M),n)$ involves the \emph{Beilinson regulator map}
\begin{equation*}
r_{\BB} : H^{i+1}_{\mathcal{M}/\OK}(M,E(n)) \otimes_\Q \R \to H^{i+1}_{\mathcal{D}}(M,E_\R(n)).
\end{equation*}

Let us briefly recall the definitions of the cohomology groups involved. The relevant motivic cohomology group is given by
\begin{equation*}
H^{i+1}_{\mathcal{M}}(M,E(n)) = p^* \bigl(H^{i+1}_{\mathcal{M}}(X,\Q(n)) \otimes E\bigr).
\end{equation*}
where $H^{i+1}_{\mathcal{M}}(X,\Q(n))$ is defined as Quillen's $K$-group $K_{2n-i-1}^{(n)}(X)$. Let $H^{i+1}_{\mathcal{M}/\OK}(M,E(n))$ be the subspace of integral elements defined by Scholl \cite{scholl:integral_elements}.

The Deligne cohomology group can be expressed as follows. Let $c \in \Gal(\CC/\R)$ denote complex conjugation. The isomorphisms $c^* : X_\sigma(\CC) \xrightarrow{\cong} X_{\overline{\sigma}}(\CC)$ together with complex conjugation on $\Q(n)=(2\pi i)^n \Q$ induce an $A$-linear involution $c_B : H^i_{\B}(M(n)) \to H^i_{\B}(M(n))$, which makes the following diagram commute
\begin{equation*}
\begin{tikzcd}
H^i_{B}(M(n)) \otimes_\Q \CC \rar{I_\infty} \dar[swap]{c_B \otimes c} & H^i_{\dR}(M(n)) \otimes_{\Q} \CC \dar{1 \otimes c} \\
H^i_{B}(M(n)) \otimes_\Q \CC \rar{I_\infty} & H^i_{\dR}(M(n)) \otimes_{\Q} \CC.
\end{tikzcd}
\end{equation*}
Let $H^i_\B(M(n))^{\pm}$ denote the subspace of $H^i_\B(M(n))$ where $c_B$ acts by $\pm 1$. The diagram above induces an isomorphism
\begin{align}
\nonumber H^i_\dR(M(n)) \otimes_\Q \R & \cong \bigl(H^i_B(M(n))^+ \otimes_\Q \R\bigr) \oplus \bigl(H^i_B(M(n))^- \otimes_\Q \R(-1)\bigr)\\
& \cong \bigl(H^i_B(M(n))^+ \otimes_\Q \R\bigr) \oplus \bigl(H^i_B(M(n-1))^+ \otimes_\Q \R\bigr).\label{HdR HB}
\end{align}
The \emph{Deligne period map} is the canonical map
\begin{equation*}
\alpha : H^i_B(M(n))^+ \otimes_\Q \R \to (H^i_\dR(M)/\Fil^n) \otimes_\Q \R.
\end{equation*}
Since the motive $M(n)$ has weight $i-2n<0$, we have
\begin{equation*}
\ker(\alpha) \subset (\Fil^n \cap \overline{\Fil}^n) H^i_B(M) \otimes \CC = \bigoplus_\sigma \bigoplus_{\substack{a+b=i\\ a,b \geq n}} H^{a,b}_\sigma(M) = 0
\end{equation*}
so that $\alpha$ is injective. The Deligne cohomology group of $M$ is then given by the cokernel of $\alpha$ :
\begin{equation}\label{HD 1}
0 \to H^i_B(M(n))^+ \otimes_\Q \R \xrightarrow{\alpha} (H^i_\dR(M)/\Fil^n) \otimes_\Q \R \to H^{i+1}_{\mathcal{D}}(M,E_\R(n)) \to 0.
\end{equation}

\begin{conjecture}[Beilinson] \label{conj1}
The regulator map $r_\BB$ is an isomorphism.
\end{conjecture}

Now the idea is that both the domain and codomain of the regulator map carry natural $A$-structures, and comparing these two $A$-structures is enough to determine the equivariant $L$-value up to an element of $Z(A)^\times$. The Deligne period map and the Beilinson regulator map are $A_\R$-linear, and (\ref{HD 1}) is an exact sequence of $A_\R$-modules. Assuming Conjecture \ref{conj1}, the exact sequence (\ref{HD 1}) together with $r_\BB$ yields a canonical element $\vartheta_\infty = \vartheta_\infty(M,i,n)$ in $K_0(A,\R)$. We may now formulate the conjecture on the $L$-value as follows.

\begin{conjecture}[Burns-Flach]\label{conj2}
Let $n > \frac{i}{2}+1$ be an integer. We have the following equality in $K_0(A,\R)$
\begin{equation}\label{conj2 eq}
\hat{\delta}(L({}_A H^i(M),n)) = \vartheta_\infty(M,i,n).
\end{equation}
\end{conjecture}

\begin{remark}
By taking norms down to $E_\R$, Conjecture \ref{conj2} implies Beilinson's conjecture for the classical $L$-value $L(H^i(M),n) \in E_\R^\times$. Note that in the case $A=E=\Q$, we have $K_0(A,\R) \cong \R^\times / \Q^\times$ and (\ref{conj2 eq}) is just a restatement of the usual conjecture.
\end{remark}

Assuming the meromorphic continuation of the equivariant $L$-function, we may reformulate Conjecture \ref{conj2} using $L$-values at integers to the left of the central point. For this we use a different $A$-structure in Deligne cohomology. Using (\ref{HdR HB}), we may also express Deligne cohomology as
\begin{equation}\label{HD 2}
0 \to \Fil^n H^i_\dR(M) \otimes_\Q \R \to H^i_B(M(n-1))^+ \otimes_\Q \R \to H^{i+1}_{\mathcal{D}}(M,E_\R(n)) \to 0
\end{equation}
where the first arrow is induced by the projection on the second factor of (\ref{HdR HB}). Assuming Conjecture \ref{conj1}, the exact sequence (\ref{HD 2}) together with $r_\BB$ yields a canonical element $\vartheta'_\infty=\vartheta'_\infty(M,i,n)$ in $K_0(A,\R)$.

Since $A$ is a semisimple algebra, we have a reduced rank morphism $\rr_A : K_0(A) \to H^0(\Spec Z(A),\Z)$ with values in the group of $\Z$-valued functions on $\Spec Z(A)$ \cite[\S 2.6, p. 510]{burns-flach}. For each embedding $\sigma$ of $E$ into $\CC$, we have a canonical morphism $\Spec Z(A_\sigma) \to \Spec Z(A)$, from which we get a morphism $ \rr_{A,\sigma} : K_0(A) \to H^0(\Spec Z(A_\sigma),\Z)$.

\begin{conjecture}[Burns-Flach]\label{conj3}
Let $n > \frac{i}{2}+1$ be an integer. For any embedding $\sigma : E \hookrightarrow \CC$, we have
\begin{equation}\label{conj3 eq}
\ord_{s=1-n} L({}_{A^{\op}} H^{2d-i}(M^*),s)^\sigma = \rr_{A,\sigma} \bigl(H^{i+1}_{\mathcal{M}/\OK}(M,E(n))\bigr).
\end{equation}
Furthermore, let $L^* \in Z(A_\R)^\times$ denote the leading term of the Taylor expansion of the $L$-series $L({}_{A^{\op}} H^{2d-i}(M^*),s)$ at $s=1-n$, defined componentwise. Then we have
\begin{equation}\label{conj3 eq2}
\hat{\delta}(L^*) = \vartheta'_\infty(M,i,n).
\end{equation}
\end{conjecture}

\begin{remark}
The reduced rank in (\ref{conj3 eq}) depends only on the realization of $M$ in Deligne cohomology if we assume Conjecture \ref{conj1}.
\end{remark}

In general these conjectures are out of reach as we cannot prove that the motivic cohomology groups are finite-dimensional. Therefore, one often uses the following weakened conjecture.

\begin{conjecture}\label{conj4}
There exists an $A$-submodule $W$ of $H^{i+1}_{\mathcal{M}/\OK}(M,E(n))$ such that
\begin{equation}\label{rBW}
r_\BB(W) \otimes_\Q \R \cong H^{i+1}_\mathcal{D}(M,E_\R(n)).
\end{equation}
Furthermore, let $\vartheta_\infty(W)$ (resp. $\vartheta'_\infty(W)$) be the element of $K_0(A,\R)$ arising from $r_\BB(W)$ by means of the exact sequence (\ref{HD 1}) (resp. (\ref{HD 2})). Then we have the equalities
\begin{align}
\hat{\delta}(L({}_A H^i(M),n)) & = \vartheta_\infty(W)\\
\hat{\delta}(L^*) & = \vartheta'_\infty(W),
\end{align}
where $L^* \in Z(A_\R)^\times$ is defined as in Conjecture \ref{conj3}.
\end{conjecture}

\begin{remark}
We may ask for a property which is stronger than (\ref{rBW}), namely that $r_\BB$ induces an isomorphism $W \otimes_\Q \R \xrightarrow{\cong} H^{i+1}_\mathcal{D}(M,E_\R(n))$.
\end{remark}

Finally, let us spell out the conjecture in the particular case of abelian varieties. Let $B$ be an abelian variety defined over a number field $K$. Consider the motive with $\Q$-coefficients $M=H^1(B)=(B,p_1,0)$, where $p_1$ the Chow-Künneth projector in degree 1. The usual $L$-function of $B/K$ is given by $L(B,s)=L(H^1(B),s)$. Let $A=\End_K(B) \otimes \Q$. The semisimple algebra $A^{\textrm{op}}$ acts on $M$, and we denote by $L({}_A B,s)=L({}_{A^{\textrm{op}}} H^1(B),s)$ the associated equivariant $L$-function. It converges for $\Re(s)>\frac32$ and takes values in $Z(A_\CC)$. Let $B^\vee$ be the dual abelian variety of $B$. The Poincaré bundle on $B \times B^\vee$ induces a canonical isomorphism $M^* \cong H^1(B^\vee)(1)$ in $\CHM_K(A)$. The (conjectural) functional equation thus relates $L({}_A B,s)$ and $L({}_{A^{\op}} B^\vee,2-s)$. Let $n \geqslant 2$ be an integer. We have isomorphisms
\begin{equation*}
H^2_{\mathcal{D}}(M,\R(n))=H^2_{\mathcal{D}}(B_\R,\R(n)) \cong \frac{H^1_\dR(B) \otimes \R}{H^1_B(B(\CC),\R(n))^+} \cong H^1_B(B(\CC),\R(n-1))^+.
\end{equation*}
Let $B \sim \prod_{i=1}^r B_i^{e_i}$ be the decomposition of $B$ into $K$-simple factors up to isogeny, and let $D_i=\End_K(B_i) \otimes \Q$. We have $A \cong \prod_{i=1}^r M_{e_i}(D_i)$ so that $Z(A) = \prod_{i=1}^r Z(D_i)$. The reduced rank of $H^1_B(B(\CC),\Q(n-1))^+$ over $A$ is the function $i \mapsto \dim(B_i)$.

\begin{conjecture}
There exists an $A^{\mathrm{op}}$-submodule $W$ of $H^2_{\mathcal{M}/\OK}(B,\Q(n))$ such that
\begin{equation}\label{rBW ab}
r_\BB(W) \otimes_\Q \R \cong H^2_\mathcal{D}(B_\R,\R(n)).
\end{equation}

Furthermore, let $\vartheta_\infty(W)$ be the element of $K_0(A^{\mathrm{op}},\R)$ arising from the exact sequence
\begin{equation}
0 \to H^1_B(B(\CC),\Q(n))^+ \otimes \R \xrightarrow{\alpha} H^1_\dR(B) \otimes \R \to r_\BB(W) \otimes \R \to 0,
\end{equation}
and let $\vartheta'_\infty(W)$ be the element of $K_0(A^{\mathrm{op}},\R)$ arising from the isomorphism
\begin{equation}
r_\BB(W) \otimes \R \xrightarrow{\cong} H^1_B(B(\CC),\Q(n-1))^+ \otimes \R.
\end{equation}
Then we have
\begin{align}
\hat{\delta}(L({}_A B,n)) & = \vartheta_\infty(W)\\
\hat{\delta}(L^*) & = \vartheta'_\infty(W),
\end{align}
where $L^* \in Z(A_\R)^\times$ denotes the leading term of the Taylor expansion of $L({}_{A^{\op}} B^\vee,s)$ at $s=2-n$.\end{conjecture}

\subsection{Base changes of Chow motives}

If $R$ is any ring and $G$ is any group acting on $R$ by ring automorphisms, the \emph{twisted group ring} $R\{G\}$ is the free $R$-module with basis $G$, endowed with the product
\begin{equation*}
\bigl(\sum_{\sigma \in G} a_\sigma \cdot \sigma\bigr) \bigl(\sum_{\tau \in G} b_\tau \cdot \tau\bigr) = \sum_{\sigma, \tau \in G} a_\sigma \sigma(b_\tau) \cdot \sigma \tau.
\end{equation*}

Let $L/K$ be a Galois extension of number fields, with Galois group $G$. There is a canonical base change functor $\CHM_K(E) \to \CHM_L(E)$ sending a Chow motive $M=(X,p,n)$ to $M_L=(X_L,p_L,n)$. In particular, we have a canonical morphism of $E$-algebras $\End_K(M) \to \End_L(M_L)$. Note that $M_L$ is a Chow motive over $L$, but we may also consider it as a Chow motive over $K$.

\begin{lem}\label{EndK ML}
For every $M \in \CHM_K(E)$, there is a canonical isomorphism $\End_K(M_L) \cong \End_L(M_L)\{G\}$.
\end{lem}

\begin{proof}
For any smooth projective $L$-variety $Y$, we have
\begin{equation*}
Y \times_K Y \cong \bigsqcup_{\sigma \in G} Y \times_L Y^\sigma
\end{equation*}
where $Y^\sigma/L$ denotes the conjugate variety. Thus we have an isomorphism of abelian groups
\begin{equation}\label{CHYY}
\CH(Y \times_K Y) \cong \bigoplus_{\sigma \in G} \CH(Y \times_L Y^\sigma).
\end{equation}
Now if $Y=X_L$ is the base change of a smooth projective $K$-variety $X$, then $Y^\sigma=Y$ so that $\CH(Y \times_K Y)$ is the direct sum of copies of $\CH(Y \times_L Y)$. The ring structure can be described as follows. For any $\sigma \in G$, let $\phi_\sigma : X_L \to X_L$ denote the $K$-automorphism induced by $\sigma$, and let $\Gamma_\sigma = \phi_\sigma^* \subset X_L \times_K X_L$ denote the transpose of the graph of $\phi_\sigma$. We have $\Gamma_\sigma \Gamma_\tau = \phi_\sigma^* \phi_\tau^* = (\phi_\tau \phi_\sigma)^* = \phi_{\sigma \tau}^* = \Gamma_{\sigma \tau}$, so we get a group morphism $\Gamma : G \to \Aut_K(H(X_L))$ where $H(X_L)$ is the total motive of $X_L$. By (\ref{CHYY}), we get
\begin{equation*}
\End_K(H(X_L)) = \bigoplus_{\sigma \in G} \End_L(H(X_L)) \cdot \Gamma_\sigma = \End_L(H(X_L)) \{G\}.
\end{equation*}
For an arbitrary $M=(X,p,n) \in \CHM_K(E)$, the idempotent $p_L \in \End_L(H(X_L)) \subset \End_K(H(X_L))$ commutes with the action of $G$, so that we get a corresponding decomposition
\begin{equation*}
\End_K(M_L) = p_L \End_K(H(X_L)) p_L = \bigl(p_L \End_L(H(X_L)) p_L\bigr) \{G\} = \End_L(M_L)\{G\}.
\end{equation*}
\end{proof}

We will also need the following lemma from non-commutative algebra.

\begin{lem}
If $A$ is a semisimple $\Q$-algebra and $G$ is a finite group acting on $A$ by $\Q$-automorphisms, then $A\{G\}$ is semisimple.
\end{lem}

\begin{proof}
Let $M$ be an arbitrary $A\{G\}$-module. Let us show that every submodule $N$ of $M$ is a direct factor. Since $A$ is semisimple, there exists an $A$-linear map $p : M \to N$ such that $p(x)=x$ for all $x \in N$. Define $p' : M \to N$ by
\begin{equation*}
p' = \frac{1}{|G|} \sum_{\sigma \in G} \sigma p \sigma^{-1}.
\end{equation*}
It is easy to check that $p'$ is $A$-linear and commutes with the action of $G$, so that $p'$ is $A\{G\}$-linear. Moreover $p'(x)=x$ for all $x \in N$, so that $N$ is a direct factor of $M$.
\end{proof}

Let $B$ be an abelian variety defined over $K$, and let $B_L = B \times_{\Spec K} \Spec L$ be its base change to $L$. Let $A=\End_L(B_L) \otimes \Q$ be the algebra of endomorphisms of $B$ defined over $L$. By Lemma \ref{EndK ML}, we have an isomorphism $\End_K(H^1(B_L)) \cong A^{\mathrm{op}} \{G\}$. Note that $A$ and $G$ commute if and only if all endomorphisms of $B_L$ are defined over $K$. We may consider the equivariant $L$-function $L({}_{A\{G\}} B_L,s)$ and formulate a conjecture on the values $L({}_{A\{G\}} B_L,n)$, $n \geqslant 2$ as in \S\ref{conj statement}. Note that this conjecture specializes to a conjecture on all Artin-twisted $L$-values $L(B \otimes \rho,n)$ for any finite-dimensional complex representation $\rho$ of $G$ and any integer $n \geqslant 2$.

\subsection{Functoriality}

In this section we recall functoriality results for the equivariant Beilinson conjecture. Note that all compatibility results below are studied and proved by Burns and Flach in the more general setting of the equivariant Tamagawa number conjecture \cite{burns-flach}. In the following results, the « equivariant Beilinson conjecture » means any of the Conjectures \ref{conj1}, \ref{conj2}, \ref{conj3}, \ref{conj4}.

As a first step, the equivariant Beilinson conjecture is clearly compatible with taking direct sums of Chow motives. We next study the behaviour of the conjecture under change of coefficients.

\begin{pro}\label{func1}
Let $E,E'$ be number fields with $E \subset E'$. Let $A$ be a finite-dimensional semisimple $E$-algebra, and let $A'=A \otimes_E E'$. Let $M=(X,p,0,\rho) \in \CHM_K(A)$ be a Chow motive, and let $M'=M \otimes_E E' \in \CHM_K(A')$. Let $i,n$ be integers such that $0 \leqslant i \leqslant 2\dim X$ and $n > \frac{i}{2}+1$. Then the equivariant Beilinson conjecture holds for $L({}_A H^i(M),n)$ if and only if it holds for $L({}_{A'} H^i(M'),n)$.
\end{pro}

\begin{proof}
The equivariant $L$-function of $H^i(M')$ is the image of the equivariant $L$-function of $H^i(M)$ under the canonical map $Z(A_\CC) \to Z(A'_\CC)$. Moreover, the regulator map associated to $(M',i,n)$ is obtained from the regulator map associated to $(M,i,n)$ by tensoring with $E'$ over $E$. Since the extended boundary map is functorial, we are thus reduced to show that the canonical map $\iota : K_0(A,\R) \to K_0(A',\R)$ is injective. We may assume that $A$ is a central simple algebra over $E$. We have a commutative diagram
\begin{equation} \label{diag K0 rel}
\begin{tikzcd}
0 \rar & K_1(A) \rar \dar & K_1(A_\R) \rar{\delta} \dar & K_0(A,\R) \dar{\iota} \rar & 0 \\
0 \rar & K_1(A') \rar & K_1(A'_\R) \rar{\delta'} & K_0(A',\R) \rar & 0.
\end{tikzcd}
\end{equation}
We may identify all the $K_1$-groups with subgroups of $Z(A'_\R)^\times$. Let $x \in K_0(A,\R)$ be in the kernel of $\iota$, and let $z \in K_1(A_\R)$ such that $\delta(z)=x$. Since $Z(A_\R)^\times \cap Z(A')^\times = Z(A)^\times$, we have $\nr_\R(z) \in Z(A)^\times$. Looking at the conditions (\ref{nr}) and (\ref{nrR}) describing the image of the reduced norm maps, we see that $z$ comes from $K_1(A)$ and thus $x=0$.
\end{proof}

\begin{pro}\label{func2}
Let $E$ be a number field, and let $\rho : A \to B$ be a morphism beween finite-dimensional semisimple $E$-algebras. Let $M=(X,p,0) \in \CHM_K(B)$ be a Chow motive, and let $\rho^* M \in \CHM_K(A)$ be the motive obtained by restricting the action to $A$. Let $i,n$ be integers such that $0 \leqslant i \leqslant 2\dim X$ and $n > \frac{i}{2}+1$. Then the equivariant Beilinson conjecture for $L({}_B H^i(M),n)$ implies the equivariant Beilinson conjecture for $L({}_{A} H^i(\rho^* M),n)$. Moreover, if $\rho$ is surjective then the converse holds.
\end{pro}

\begin{proof}
The map $\rho$ induces an exact functor from the category of finitely generated $B$-modules to the category of finitely generated $A$-modules, which in turns induces maps $\rho^*$ on $K$-groups. Assume the equivariant Beilinson conjecture for $L({}_B H^i(M),n)$. Let ${}_B \vartheta_\infty$ be the corresponding element of $K_0(B,\R)$, and let ${}_A \vartheta_\infty = \rho^*({}_B \vartheta_\infty)$. By Lemma \ref{lem K1 cartesien}, we have isomorphisms $K_1(A_\CC) \cong Z(A_\CC)^\times$ and $K_1(B_\CC) \cong Z(B_\CC)^\times$. We use these to define a norm map $\rho^* : Z(B_\CC)^\times \to Z(A_\CC)^\times$. By construction of the equivariant $L$-function, we then have $\rho^* (L({}_B H^i(M),s))=L({}_A H^i(\rho^* M),s)$ (see \cite[Thm 4.1]{burns-flach}). Taking invariants under $\Gal(\CC/\R)$, we also have a map $\rho^* : Z(B_\R)^\times \to Z(A_\R)^\times$, and we are left to show that $\rho^*$ commutes with the extended boundary map, in other words that $\rho^* \circ \hat{\delta}_B = \hat{\delta}_A \circ \rho^*$. This identity is true on $K_1(B_\R)$ because the boundary map is functorial, and it is true on $Z(B)^\times$ because $\rho^*(Z(B)^\times) \subset Z(A)^\times$.

Assume $\rho$ is surjective. By the discussion above, it suffices to prove that $\rho^* : K_0(B,\R) \to K_0(A,\R)$ is injective. Since $A$ is semisimple, we must have an isomorphism $A \cong B \times B'$ such that $\rho$ becomes the canonical projection. Then $K_0(A,\R) \cong K_0(B,\R) \oplus K_0(B',\R)$ and the result is clear.
\end{proof}

\begin{pro}\label{func3}
Let $E$ be a number field, and let $A$ be a finite-dimensional semisimple $E$-algebra. Let $e$ be a nonzero idempotent of $A$, and let $A'=eAe$. Let $M=(X,p,0) \in \CHM_K(A)$ be a Chow motive, and let $i,n$ be integers such that $0 \leqslant i \leqslant 2\dim X$ and $n > \frac{i}{2}+1$. If the equivariant Beilinson conjecture holds for $L({}_A H^i(M),n)$, then it holds for $L({}_{A'} H^i(e(M)),n)$.
\end{pro}

\begin{proof}
The algebra $A'$ is semisimple (see \cite[\S 9, Exerc. 10d, p. 162]{bourbaki-alg8}). We have an exact functor $e^*$ sending a finitely generated $A$-module $V$ to the $A'$-module $V'=e(V)$. It induces maps $e^* : K_1(A_\R) \to K_1(A'_\R)$ and $e^* : K_0(A,\R) \to K_0(A',\R)$. Moreover, we have a morphism of $E$-algebras $e^* : Z(A) \to Z(A')$ sending $x$ to $exe$. By definition of the reduced norm map, the diagram
\begin{equation}
\begin{tikzcd}
K_1(A_\R) \rar{\nr_\R} \dar{e^*} & Z(A_\R)^\times \dar{e^*} \\
K_1(A'_\R) \rar{\nr'_\R} & Z(A'_\R)^\times
\end{tikzcd}
\end{equation}
is commutative. It follows that $e^*$ commutes with the extended boundary maps. By definition of the equivariant $L$-function \cite[\S 4.1]{burns-flach}, we have $e^*(L({}_A H^i(M),s)) = L({}_{A'} H^i(e(M)),s)$. Assume the equivariant Beilinson conjecture for $L({}_A H^i(M),n)$, and let ${}_A \vartheta_\infty$ be the corresponding element of $K_0(A,\R)$. Applying $e^*$ to all objects appearing in the Beilinson regulator map, we see that the element of $K_0(A',\R)$ associated to the regulator map for ${}_{A'} H^i(e(M))$ is simply $e^*({}_A \vartheta_\infty)$. Thus the equivariant conjecture for $L({}_{A'} H^i(e(M)),n)$ holds.
\end{proof}

\section{Modular abelian varieties} \label{sec modular abvar}

In this section and \S \ref{sec modular curves}, we fix a newform $f=\sum_{n \geq 1} a_n q^n$ of weight $2$ on $\Gamma_1(N)$ \emph{without complex multiplication}. Let $K_f \subset \CC$ be the number field generated by the Fourier coefficients of $f$.

Let $A_f/\Q$ be the modular abelian variety attached to $f$. It is defined as the quotient $J_1(N)/I_f J_1(N)$, where $J_1(N)$ is the Jacobian of the modular curve $X_1(N)$, and $I_f$ is the annihilator of $f$ in the Hecke algebra. There is a natural isomorphism $K_f \cong \End_\Q(A_f) \otimes \Q$, which shows that $A_f$ is simple over $\Q$. In general, the abelian variety $A_f$ is not absolutely simple. We first recall a standard result on the simple factors of $A_f$ over a given extension of $\Q$.

Fix a subfield $F$ of $\Qb$. Let $X=\End_F(A_f) \otimes \Q$ be the endomorphism algebra of $(A_f)_F$. The following theorem was proved by Ribet \cite[Thm 5.1]{ribet:twists} in the case $F=\Qb$. The general case follows rather easily from this case.

\begin{thm}\label{thm BfF}
\begin{enumerate}
\item[(a)] The center $k$ of $X$ is a subfield of $K_f$.
\item[(b)] The dimension of $X$ over $k$ is $[K_f:k]^2$.
\item[(c)] The abelian variety $A_f$ is isogenous over $F$ to the power of a simple abelian variety $B_{f,F}/F$.
\item[(d)] The abelian variety $B_{f,F}$ is unique up to $F$-isogeny. Moreover, if $F/\Q$ is Galois, then $B_{f,F}$ is $F$-isogenous to all its $\Gal(F/\Q)$-conjugates.
\end{enumerate}
\end{thm}

\begin{proof}
Since $f$ doesn't have complex multiplication, the abelian variety $(A_f)_{\Qb}$ has no abelian subvariety of CM-type. This implies that $K_f$ is its own commutant in $\End_{\Qb}(A_f) \otimes \Q$ (see the proof of \cite[Prop. 5.2]{ribet}), which proves $(a)$. Now $X$ is a central simple algebra over $k$, and $K_f$ is a (semisimple) maximal commutative subalgebra of $X$, so that $[X:k]=[K_f:k]^2$ by \cite[\S 14, N°6, Prop. 3]{bourbaki-alg8}, which proves $(b)$. Moreover $k$ being a field means precisely that $A_f$ is $F$-isogenous to the power of a simple abelian variety over $F$, which proves $(c)$. Finally $(d)$ follows from the unicity of decomposition of $(A_f)_F$ into simple factors up to isogeny, together with the fact that $A_f$ is defined over $\Q$.
\end{proof}

\begin{remark}
In the particular case where $F/\Q$ is Galois and $B_{f,F}$ is an elliptic curve, Theorem \ref{thm BfF}$(d)$ says precisely that $B_{f,F}$ is a $\Q$-curve completely defined over $F$ in the terminology of \cite[p. 286]{quer:Qcurves}.
\end{remark}

It is known that the minimal number field over which all endomorphisms of $A_f$ are defined is a finite abelian extension of $\Q$ \cite[Prop. 2.1]{gonzalez-lario}.

In the following, we fix a finite abelian extension $F/\Q$. We show that the $L$-function of $B_{f,F}$ can be expressed as a product of twists of $L$-functions of conjugates of $f$. Note that $B_{f,F}$ is defined only up to $F$-isogeny, but it makes sense to speak of its $L$-function.

Let $V_\ell$ be the Tate module of $A_f$ with coefficients in $\Q_\ell$. It carries an action of $G_\Q=\Gal(\Qb/\Q)$. After choosing an isomorphism $\overline{\Q_\ell} \cong \CC$, we have a decomposition
\begin{equation}\label{decomp_Vell}
\overline{V_\ell} := V_\ell \otimes_{\Q_\ell} \overline{\Q_\ell} \cong \prod_{\sigma : K_f \hookrightarrow \CC} V_{f^\sigma}
\end{equation}
where $V_{f^\sigma}$ denotes the $2$-dimensional $\overline{\Q_\ell}$-representation of $G_\Q$ associated to $f^\sigma$. This decomposition is compatible with the action of $K_f$, where $K_f$ acts on $V_{f^\sigma}$ through $\sigma$. Let $G = \Gal(F/\Q)$, and let $\hat{G}$ be the group of complex-valued characters of $G$. We will identify elements of $\hat{G}$ with Dirichlet characters in the usual way.

\begin{lem}\label{emb_equiv_F}
Let $\sigma,\tau : K_f \hookrightarrow \CC$. The following conditions are equivalent :
\begin{enumerate}
\item[(a)] The restrictions of $V_{f^\sigma}$ and $V_{f^\tau}$ to $G_F=\Gal(\Qb/F)$ are isomorphic.
\item[(b)] There exists a character $\chi \in \hat{G}$ such that $f^\tau = f^\sigma \otimes \chi$.
\item[(c)] We have $\sigma |_k = \tau|_k$.
\end{enumerate}
If these conditions are satisfied, then the character $\chi$ in (b) is unique.
\end{lem}

\begin{proof}
See the proof of \cite[Thm 4.4]{ribet:twists}.
\end{proof}

Consider the equivalence relation on $\Hom(K_f,\CC)$ given by Lemma \ref{emb_equiv_F}, namely $\sigma \sim \tau \Leftrightarrow \sigma |_k = \tau |_k$. Fix a system $\Sigma$ of representatives of $\Hom(K_f,\CC)/\sim$. We have $|\Sigma| = [k:\Q]$.

\begin{pro}\label{LBfF-pro}
Let $D=\End_F(B_{f,F}) \otimes \Q$, and let $t =[D:k]^{1/2}$ be the degree of $D$. We have
\begin{equation}\label{LBfF-formula}
L(B_{f,F}/F,s) = \prod_{\sigma \in \Sigma} \prod_{\chi \in \hat{G}} L(f^\sigma \otimes \chi,s)^t.
\end{equation}
\end{pro}

\begin{proof}
Write $A_f \sim_F B_{f,F}^n$, so that $X \cong M_n(D)$. By Theorem \ref{thm BfF}$(b)$, we know that $[K_f:k]=nt$. By Lemma \ref{emb_equiv_F}, the $\overline{\Q_\ell}$-Tate module of $B_{f,F}$ is isomorphic as a $G_F$-module to $\prod_{\sigma \in \Sigma} V_{f^\sigma}^t$. For a given embedding $\sigma : K_f \hookrightarrow \CC$, we have
\begin{equation*}
\Ind_{G_F}^{G_\Q} (V_{f^\sigma} |_{G_F}) \cong \bigoplus_{\chi \in \hat{G}} V_{f^\sigma} \otimes \chi.
\end{equation*}
Taking $L$-functions of both sides, and using Artin formalism, we get
\begin{equation*}
L(V_{f^\sigma} |_{G_F},s) = \prod_{\chi \in \hat{G}} L(f^\sigma \otimes \chi,s).
\end{equation*}
The formula for $L(B_{f,F}/F,s)$ follows.
\end{proof}

Conversely, we have the following result by Guitart and Quer \cite{guitart-quer,guitart-quer-jacobian}.

\begin{thm}\label{strongly modular}
Let $A$ be an abelian variety over a number field $F$ such that $L(A/F,s)$ is a product of $L$-functions of newforms of weight $2$ without complex multiplication. Then the extension $F/\Q$ is abelian, and there exist newforms $f_1,\ldots,f_r$ of weight $2$ without complex multiplication such that $A$ is $F$-isogenous to $B_{f_1,F} \times \cdots \times B_{f_r,F}$.
\end{thm}

\begin{proof}
Let $B=\Res_{F/\Q} A$ be the restriction of scalars of $A$. Let $f_1,\ldots,f_r$ be newforms of weight $2$ such that $L(A/F,s)=L(B/\Q,s)=L(f_1,s) \cdots L(f_r,s)$. By the proof of \cite[Prop 2.3]{guitart-quer}, the abelian variety $B$ is $\Q$-isogenous to $A_{f_1}^{n_1} \times \cdots \times A_{f_r}^{n_r}$ for some integers $n_1,\ldots,n_r \geqslant 0$, and by the proof of \cite[Prop 2.2]{guitart-quer-jacobian}, the extension $F/\Q$ is abelian. Let $C$ be a $F$-simple factor of $A$. The abelian variety $D=\Res_{F/\Q} C$ is a factor of $B$, thus is also $\Q$-isogenous to $A_{f_1}^{m_1} \times \cdots \times A_{f_r}^{m_r}$ for some integers $0 \leqslant m_i \leqslant n_i$. Moreover $C$ is a $F$-simple factor of $D_F$ and thus a factor of $(A_f)_F$ for some newform $f$ without CM.
\end{proof}

The abelian varieties whose $L$-functions are products of $L$-functions of newforms of weight 2 are called \emph{strongly modular} in \cite{guitart-quer}. By Theorem \ref{strongly modular}, every non-CM strongly modular abelian variety over a number field is a $\Q$-variety, in the sense that it is isogenous to all its Galois conjugates. In the particular case of elliptic curves, this gives the following result.

\begin{cor}\label{strongly modular 2}
Let $E$ be an elliptic curve without complex multiplication over a number field $F$ such that $L(E/F,s)$ is a product of $L$-functions of newforms of weight $2$. Then the extension $F/\Q$ is abelian, and there exist a newform $f$ of weight $2$ without complex multiplication such that $E$ is $F$-isogenous to $B_{f,F}$. In particular $E$ is a $\Q$-curve completely defined over $F$.
\end{cor}

\begin{remark}
As was pointed out to me by Xavier Guitart, Theorem \ref{strongly modular} and Corollary \ref{strongly modular 2} do not hold in the CM case. As an example, let $K=\Q(\sqrt{-23})$ and let $H$ be the Hilbert class field of $K$. Let $E$ be a $\Q$-curve over $H$ with complex multiplication by $K$. Then $E$ can be defined over the cubic number field $\Q(j(E))$. By a result of Nakamura \cite[\S 5, Thm 3]{nakamura}, the restriction of scalars $B=\Res_{\Q(j(E))/\Q} E$ is an abelian variety of $\GL_2$-type, thus $E$ is strongly modular, but the extension $\Q(j(E))/\Q$ is not Galois.
\end{remark}

It was predicted by Serre that the $\Q$-curves are precisely the elliptic curves which arise as quotients of $J_1(N)$ over $\Qb$. This is now a theorem thanks to the work of Ribet \cite{ribet} and the proof of Serre's modularity conjecture due to Khare-Wintenberger (see \cite[Thm 7.2]{khare}). It follows that every $\Q$-curve $E/\Qb$ is isogenous over $\Qb$ to $B_{f,\Qb}$ for some newform $f$ of weight $2$. It seems an interesting question to determine a minimal field of definition for this isogeny in terms of the arithmetic of $E$. By Corollary \ref{strongly modular 2}, every non-CM strongly modular $\Q$-curve $E/F$ is completely defined over $F$. The converse is not true, even if $F/\Q$ is abelian : see the introduction of \cite{guitart-quer} for a counterexample with $F=\Q(\sqrt{-2},\sqrt{-3})$. However, if $F$ is a quadratic field, then every non-CM $\Q$-curve completely defined over $F$ is strongly modular, so that our results will apply to these $\Q$-curves. In the general case, necessary and sufficient conditions for strong modularity in terms of splittings of $2$-cocycles are worked out in \cite[Thm 5.3, Thm 5.4]{guitart-quer}.

\section{Modular curves in the adelic setting} \label{sec modular curves}

\subsection{Notations and standard results}

Let us recall the notations of \cite[\S 4]{brunault:LEF}. Let $\A_f$ be the ring of finite adèles of $\Q$. To any compact open subgroup $K$ of $\GL_2(\A_f)$ is associated a smooth projective modular curve $\overline{M}_K$ over $\Q$, whose set of complex points $\overline{M}_K(\CC)$ is the compactification of the Riemann surface $\GL_2(\Q) \backslash (\mathfrak{h}^{\pm} \times \GL_2(\A_f)) / K$. There are natural projections $\pi_{K',K} : \overline{M}_{K'} \to \overline{M}_K$ for any compact open subgroups $K' \subset K$ of $\GL_2(\A_f)$. For any $g \in \GL_2(\A_f)$, there is a canonical isomorphism $g : \overline{M}_K \xrightarrow{\cong} \overline{M}_{g^{-1}Kg}$, given at the level of complex points by $(\tau,h) \mapsto (\tau,hg)$.

The Hecke algebra $\tilde{\TT}_K$ is the space of functions $K \backslash \GL_2(\A_f)/K \to \Q$ with finite support, equipped with the convolution product \cite{cartier:corvallis}. We may identify $\tilde{\TT}_K$ with its image in the $\Q$-algebra of finite correspondences on $\overline{M}_K$ by sending the characteristic function of $KgK$ to the correspondence $\tilde{T}(g)=\tilde{T}(g)_K$ defined by the diagram
\begin{equation}\label{Ttildeg}
\begin{tikzcd}
 & \overline{M}_{K \cap g^{-1}Kg} \dlar[swap]{\pi} \drar{\pi'} \\
\overline{M}_K \ar[dashed]{rr}{\tilde{T}(g)} & & \overline{M}_K
\end{tikzcd}
\end{equation}
where $\pi = \pi_{K \cap g^{-1} K g, K}$ and $\pi' = \pi_{gKg^{-1} \cap K,K} \circ g^{-1}$.

The space $\Omega^1(\overline{M}_K)$ carries a natural structure of left $\tilde{\TT}_K$-module, and we denote by $\TT_K$ the image of $\tilde{\TT}_K$ in $\End_{\Q} (\Omega^1(\overline{M}_K))$. We denote by $T(g)=T(g)_K$ the canonical image of $\tilde{T}(g)$ in $\TT_K$. Using notations of (\ref{Ttildeg}), we have $T(g) = \pi'_* \circ \pi^*$.

The ring $\tilde{\TT}_K$ also acts from the left on $H^1(\overline{M}_K(\CC),\Q)$, and this action factors through $\TT_K$. In fact, Poincaré duality induces a perfect bilinear pairing
\begin{equation} \label{poincare duality}
\langle \cdot,\cdot \rangle : H^1(\overline{M}_K(\CC),\R)^- \times \bigl(\Omega^1(\overline{M}_K) \otimes \R\bigr) \to \R
\end{equation}
satisfying $\langle \tilde{T}(g) \eta, \omega \rangle = \langle \eta, \tilde{T}(g^{-1}) \omega \rangle$ for every $g \in \GL_2(\A_f)$, $\eta \in H^1(\overline{M}_K(\CC),\R)^-$ and $\omega \in \Omega^1(\overline{M}_K) \otimes \R$.

Let us define $\Omega = \varinjlim_{K} \Omega^1(\overline{M}_K) \otimes \Qb$, where the direct limit is taken with respect to the pull-back maps $\pi_{K',K}^*$. This space carries a natural $\GL_2(\A_f)$-action, and for any $K$ we have $\Omega^K = \Omega^1(\overline{M}_K) \otimes \Qb$. The space $\Omega$ decomposes as a direct sum of irreducible admissible representations $\Omega(\pi)$ of $\GL_2(\A_f)$. Let $\Pi(K)$ be the set of those representations $\pi$ satisfying $\Omega(\pi)^K \neq \{0\}$. We have a direct sum decomposition
\begin{equation}\label{dec Omega1}
\Omega^1(\overline{M}_K) \otimes \Qb = \bigoplus_{\pi \in \Pi(K)} \Omega(\pi)^K
\end{equation}
where the $\Omega(\pi)^K$ are pairwise non-isomorphic simple $\TT_K \otimes \Qb$-modules \cite[p. 393]{langlands}.

\begin{lem}\label{TK semisimple}
The natural map
\begin{equation*}
\TT_K \otimes \Qb \to \prod_{\pi \in \Pi(K)} \End_{\Qb}(\Omega(\pi)^K)
\end{equation*}
is an isomorphism. In particular $\TT_K$ is a semisimple algebra.
\end{lem}

\begin{proof}
The above map is injective by definition of $\TT_K$. The surjectivity follows from Burnside's Theorem \cite[\S 5, N°3, Cor. 1 of Prop. 4, p. 79]{bourbaki-alg8}. The algebra $\TT_K \otimes \Qb$, being a product of matrix algebras over $\Qb$, is semisimple. This implies that $\TT_K$ is semisimple \cite[\S 12, N°7, Cor. 2 a), p. 218]{bourbaki-alg8}.
\end{proof}

As a consequence of Lemma \ref{TK semisimple}, note that for each $\pi \in \Pi(K)$, the center $Z(\TT_K)$ acts on $\Omega(\pi)^K$ through a character $\theta_{\pi,K} : Z(\TT_K) \to \Qb$.

Let $p$ be a prime number, and let $\varpi_p$ be the element of $\A_f^\times$ whose component at $p$ is equal to $p$, and whose other components are equal to $1$. The Hecke operator $\tilde{T}(p) = \tilde{T}(p)_K \in \tilde{\TT}_K$ is defined as the characteristic function of the double coset $K \begin{pmatrix} \varpi_p & 0 \\ 0 & 1 \end{pmatrix} K$, and the Hecke operator $\tilde{T}(p,p)=\tilde{T}(p,p)_K \in \tilde{\TT}_K$ is defined as the characteristic function of $K \begin{pmatrix} \varpi_p & 0 \\ 0 & \varpi_p \end{pmatrix}$. We let $T(p)=T(p)_K$ and $T(p,p)=T(p,p)_K$ be their respective images in $\TT_K$. If $p$ doesn't divide the level of $K$, meaning that $K$ contains $\GL_2(\Z_p)$ (this happens for all but finitely many $p$), then $\tilde{T}(p)$ and $\tilde{T}(p,p)$ belong to the center of $\tilde{\TT}_K$. In this case $T(p)$ and $T(p,p)$ act by scalar multiplication on each $\Omega(\pi)^K$.

\subsection{Base changes of Hecke correspondences}\label{hecke base change}

In this subsection, we assume that $\det(K)=\hat{\Z}^\times$, which means that $\overline{M}_K$ is geometrically connected.

Let $F$ be a finite abelian extension of $\Q$, with Galois group $G=\Gal(F/\Q)$. Let $U_F$ be the subgroup of $\hat{\Z}^\times$ corresponding to $F$ by abelian class field theory. We have an isomorphism $\hat{\Z}^\times /U_F \cong G$. Let us define
\begin{equation*}
K_F = \{g \in K : \det(g) \in U_F\}.
\end{equation*}
The determinant map induces an isomorphism $K/K_F \cong G$. The modular curve $\overline{M}_{K_F}$ is canonically isomorphic to the base change $\overline{M}_K \otimes_{\Q} F$. The group $G$ acts on the right on $\Spec F$ and $\overline{M}_{K_F}$. This induces a left action of $G$ on $\Omega^1(\overline{M}_{K_F})$. The action of an element $\sigma \in G$ on $\Omega^1(\overline{M}_{K_F})$ coincides with $T(g)_{K_F}$, where $g$ is any representative of $\sigma$ in $K$.

Let $\delta : \overline{M}_{K_F} \to \Spec F$ be the structural morphism. Let $T=(X,\alpha,\beta)$ be a finite correspondence on $\overline{M}_{K_F}$, defined by the diagram
\begin{equation}
\begin{tikzcd}
 & X \dlar[swap]{\alpha} \drar{\beta} \\
\overline{M}_{K_F} \ar[dashed]{rr}{T} & & \overline{M}_{K_F}.
\end{tikzcd}
\end{equation}
There exists a unique element $\sigma \in G$ such that $\delta \circ \beta = \sigma^* \circ \delta \circ \alpha$. We say that $T$ is \emph{defined over $F$} if $\sigma=\id_G$, which amounts to say that $\delta \circ \alpha = \delta \circ \beta$.

\begin{lem}
Let $g \in \GL_2(\A_f)$. The correpondence $\tilde{T}(g)$ on $\overline{M}_{K_F}$ is defined over $F$ if and only if $\det(g) \in \Q_{>0} \cdot U_F$.
\end{lem}

We denote by $\tilde{\TT}'_{K_F}$ the subalgebra of $\tilde{\TT}_{K_F}$ generated by those correspondences $\tilde{T}(g)_{K_F}$ which are defined over $F$. Let $\TT'_{K_F}$ be the canonical image of $\tilde{\TT}'_{K_F}$ in $\TT_{K_F}$. The elements of $\TT'_{K_F}$ are precisely those elements of $\TT_{K_F}$ which are $F$-linear endomorphisms of $\Omega^1(\overline{M}_{K_F}) \cong \Omega^1(\overline{M}_K) \otimes F$, and we have an isomorphism
\begin{equation}\label{decompT}
\TT_{K_F} =  \TT'_{K_F} \{G\}.
\end{equation}

We now restrict to the case
\begin{equation*}
K = K_1(N) = \biggl\{g \in \GL_2(\Zhat) : g \equiv \begin{pmatrix} * & * \\ 0 & 1 \end{pmatrix} \pmod{N}\biggr\}.
\end{equation*}
The associated modular curves are $\overline{M}_{K_1(N)} = X_1(N)$ and $\overline{M}_{K_1(N)_F} = X_1(N)_F$.

Let us recall the relation between Hecke operators on $X_1(N)$ and $X_1(N)_F$. Define the base change morphism $\nu_F : \End_\Q(\Omega^1(X_1(N))) \to \End_F(\Omega^1(X_1(N)) \otimes F)$ by $\nu_F(T) = T \otimes \id_F$. Fix an integer $m \geqslant 1$ such that $F \subset \Q(\zeta_m)$. For any element $\alpha \in (\Z/m\Z)^\times$, let $\sigma_\alpha$ denote its canonical image in $G$.

The following lemma was proved in \cite[Lemma 13]{brunault:LEF}.

\begin{lem}\label{lem nuF}
For any prime $p$ not dividing $Nm$, we have
\begin{align}
\label{nuF Tp}\nu_F \bigl(T(p)_{K_1(N)}\bigr) & = T(p)_{K_1(N)_F} \cdot \sigma_p\\
\label{nuF Tpp}\nu_F \bigl(T(p,p)_{K_1(N)}\bigr) & = T(p,p)_{K_1(N)_F} \cdot \sigma_p^2.
\end{align}
\end{lem}

Now let $f$ be a newform of weight $2$ on $\Gamma_1(N)$. Fix an embedding $\sigma : K_f \hookrightarrow \CC$ and a character $\chi \in \hat{G}$, and let $\pi (f^\sigma \otimes \chi)$ be the automorphic representation of $\GL_2(\A_f)$ associated to the newform $f^\sigma \otimes \chi$. We have $\pi (f^\sigma \otimes \chi) \cong \pi(f^\sigma) \otimes (\tilde{\chi} \circ \det)$, where $\tilde{\chi} : \A_f^\times/\Q_{>0} \to \CC^\times$ denotes the adèlization of $\chi$, sending $\varpi_p$ to $\chi(p)$ for every prime $p$ not dividing $m$. Since $\pi(f^\sigma) \in \Pi(K_1(N))$, we have $\pi(f^\sigma \otimes \chi) \in \Pi(K_1(N)_F)$.

The following lemma was proved in \cite[Lemma 15]{brunault:LEF}.

\begin{lem}\label{lem theta}
Let $\sigma : K_f \hookrightarrow \CC$ and $\chi \in \hat{G}$. For any prime $p$ not dividing $Nm$, the operator $T(p)_{K_1(N)_F}$ (resp. $T(p,p)_{K_1(N)_F}$) acts as $\sigma(a_p) \chi(p)$ (resp. $\chi(p)^2$) on $\Omega(\pi(f^\sigma \otimes \chi))^{K_1(N)_F}$.
\end{lem}

\subsection{Modularity of endomorphism algebras} \label{modular endo}

In this section, we show that every endomorphism of $A_f$ defined over an abelian extension of $\Q$ is of automorphic origin. This is the main technical ingredient in order to apply Beilinson's theorem on modular curves. That all endomorphisms of $A_f$ over $\Qb$ are modular was proved by Ribet \cite{ribet:twists} using a construction of Shimura \cite{shimura:modular_jacobian} (see also the work of Momose \cite{momose}, Brown-Ghate \cite{brown-ghate}, Ghate-Gonz{\'a}lez-Jim{\'e}nez-Quer \cite{ghate-gonzalez-quer}, Gonz{\'a}lez-Lario \cite{gonzalez-lario}). Our approach is different in that we study endomorphisms defined over a given abelian extension of $\Q$. Moreover, our statement and proof are completely automorphic and don't involve explicit computation of Hecke operators.

In this section, we fix a finite abelian extension $F$ of $\Q$. Let $\Omega_{N,F} = \Omega^1(X_1(N)_F) \cong \Omega^1(X_1(N))_F$. In order to ease notations, let $\TT_{N,F} = \TT_{K_1(N)_F} \subset \End_\Q(\Omega_{N,F})$ and $\TT'_{N,F} = \TT'_{K_1(N)_F} \subset \End_F(\Omega_{N,F})$. By (\ref{decompT}) we have an isomorphism $\TT_{N,F} \cong \TT'_{N,F}\{G\}$.

\begin{lem}
There is a commutative diagram
\begin{equation}\label{cd EndJ1N}
\begin{tikzcd}
\TT'_{N,F} \arrow{r}{\rho'} \arrow[hook]{d} & \End_F(J_1(N))^{\mathrm{op}} \otimes \Q \arrow[hook]{r} \arrow[hook]{d} & \End_F(\Omega_{N,F}) \arrow[hook]{d} \\
\TT_{N,F} \arrow{r}{\rho} & \End_F(J_1(N))^{\mathrm{op}} \otimes \Q\{G\} \arrow[hook]{r} & \End_\Q(\Omega_{N,F})
\end{tikzcd}
\end{equation}
such that for any $T \in \TT'_{N,F}$, we have $\rho'(T)^* = T$ and for any $\sigma \in G$, we have $\rho(\sigma) = \sigma$.
\end{lem}

\begin{proof}
The cotangent space of $J_1(N)_F$ at the origin is given by $\Omega^1(J_1(N))_F$ and can be identified canonically with $\Omega_{N,F}$. We define the map $\End_F(J_1(N)) \to \End_F(\Omega_{N,F})$ by sending an endomorphism $\varphi$ of $J_1(N)_F$ to its cotangent map $\Cot(\varphi)$ at the origin. If $\tilde{T}$ is a finite correspondence on $X_1(N)_F$ defined over $F$, and $T$ is the canonical image of $\tilde{T}$ in $\End_F (\Omega_{N,F})$, then by definition of the Jacobian variety, there is a unique endomorphism $\varphi(\tilde{T}) \in \End_F (J_1(N)) \otimes \Q$ such that the $\Cot(\varphi(\tilde{T})) = T$. In particular, the restriction of the map $\tilde{T} \mapsto \varphi(\tilde{T})$ to $\tilde{\TT}'_{N,F}$ factors through $\TT'_{N,F}$. This defines the map $\rho'$ of (\ref{cd EndJ1N}). We define $\rho$ by extending linearly $\rho'$ using $\TT_{N,F} \cong \TT'_{N,F}\{G\}$.
\end{proof}

We next give a criterion for an endomorphism of $J_1(N)$ to induce an endomorphism of $A_f$. Let $\pi : J_1(N) \to A_f$ denote the canonical projection, and let $\pi_F : J_1(N)_F \to (A_f)_F$ be its base change to $F$. Let $\Omega_{f,F} = \Omega^1(A_f)_F$. We may and will identify $\Omega_{f,F}$ with its image in $\Omega_{N,F}$ by means of the canonical injection $\pi_F^* : \Omega^1(A_f)_F \to \Omega^1(J_1(N))_F$. The following result is classical.

\begin{lem}\label{lem EndAf}
Let $T$ be an element of $\TT'_{N,F}$. Then $\rho'(T)$ induces an element of $\End_F(A_f) \otimes \Q$ if and only if $T$ leaves stable $\Omega_{f,F}$.
\end{lem}

\begin{proof}
Since $A_f=J_1(N)/I_f J_1(N)$, we have an exact sequence
\begin{equation*}
0 \to \Lie (I_f J_1(N)) \to \Lie (J_1(N)) \to \Lie (A_f) \to 0.
\end{equation*}
Base changing to $F$, we get an exact sequence
\begin{equation*}
0 \to \Lie (I_f J_1(N))_F \to \Lie (J_1(N))_F \to \Lie (A_f)_F \to 0.
\end{equation*}
The dual exact sequence is
\begin{equation*}
0 \to \Omega_{f,F} \to \Omega_{N,F} \to \Omega^1(I_f J_1(N))_F \to 0.
\end{equation*}
Let $D \in \End_F (\Lie (J_1(N))_F)$ be the differential of $\rho'(T)$ at the origin. The operators $T$ and $D$ are dual to each other. Then $\rho'(T)$ induces an endomorphism of $(A_f)_F$ if and only if $D$ leaves stable $\Lie (I_f J_1(N))_F$, which means exactly that $T$ leaves stable $\Omega_{f,F}$.
\end{proof}

As a next step, we determine how $A_f$ interacts with the Hecke algebra. Fix an embedding of $\Qb$ into $\CC$. For any $\sigma : K_f \hookrightarrow \CC$, the differential form $\omega_{f^{\sigma}}=2\pi i f^{\sigma}(z) dz$ defines an element of $\Omega^1(X_1(N)) \otimes \Qb$, and the elements $(\omega_{f^\sigma})_{\sigma : K_f \hookrightarrow \CC}$ form a $\Qb$-basis of $\Omega^1(A_f) \otimes \Qb$. By the normal basis theorem, the $\Qb$-vector space $F \otimes \Qb$ splits into $\Qb$-lines $(L_\chi)_{\chi \in \hat{G}}$ such that $\sigma \in G$ acts as $\overline{\chi}(\sigma)$ on $L_\chi$.

\begin{pro}\label{pro decomp Omegaf}
We have a direct sum decomposition 
\begin{equation}\label{eq decomp Omegaf}
\Omega_{f,F} \otimes_{\Q} \Qb = \bigoplus_{\substack{\sigma : K_f \hookrightarrow \CC \\ \chi \in \hat{G}}}  \omega_{f^\sigma} \cdot L_\chi
\end{equation}
and for every $\sigma : K_f \hookrightarrow \CC$ and $\chi \in \hat{G}$, we have $\omega_{f^\sigma} \cdot L_\chi \subset \Omega(\pi(f^\sigma \otimes \chi))$.
\end{pro}

\begin{proof}
The decomposition (\ref{eq decomp Omegaf}) follows from the equality $\Omega_{f,F} \otimes \Qb = \Omega^1(A_f) \otimes F \otimes \Qb$. Let $L=\omega_{f^\sigma} \cdot L_\chi$. Let $p$ be a prime not dividing $Nm$. We know that $T(p)_{X_1(N)}(\omega_{f^\sigma})=\sigma(a_p) \omega_{f^\sigma}$. It follows that $\nu_F(T(p)_{X_1(N)})$ acts as $\sigma(a_p)$ on $L$. Moreover $\sigma_p$ acts as $\overline{\chi}(p)$ on $L$. By Lemma \ref{lem nuF}, we deduce that $T(p)_{X_1(N)_F}$ acts as $\sigma(a_p) \chi(p)$ on $L$. Similarly $T(p,p)_{X_1(N)_F}$ acts as $\chi(p)^2$ on $L$. The result now follows from Lemma \ref{lem theta} together with the multiplicity one theorems \cite{piatetski-shapiro}.
\end{proof}

\begin{pro}\label{pro ef}
There exists an idempotent $e_f \in \TT_{N,F}$ whose image is precisely $\Omega_{f,F}$.
\end{pro}

\begin{proof}
By Galois descent, it is sufficient to prove the existence of an idempotent $e_f \in \TT_{N,F} \otimes \Qb$ whose image is $\Omega_{f,F} \otimes \Qb$. This follows from Lemma \ref{TK semisimple} and Proposition \ref{pro decomp Omegaf}.
\end{proof}

\begin{remark}
Let $\iota : I_f J_1(N) \to J_1(N)$ be the canonical inclusion and consider the dual map $\iota^\vee : J_1(N) = J_1(N)^\vee \to (I_f J_1(N))^\vee$. Since the map $(\pi,\iota^\vee) : J_1(N) \to A_f \times (I_f J_1(N))^\vee$ is an isogeny, there exists a canonical projector $e_f^{\textrm{can}} \in \End_\Q(J_1(N)) \otimes \Q$ with image $A_f$. It seems reasonable to hope that $e_f^{\textrm{can}}$ belongs to the image of $\rho'$ in diagram (\ref{cd EndJ1N}), but I haven't tried to prove this.
\end{remark}

Now, let us consider the semisimple algebra $\TT_{f,F} =e_f \TT_{N,F} e_f$. It leaves stable $\Omega_{f,F}$, so that by Lemma \ref{lem EndAf}, we have an induced map $\rho_f : \TT_{f,F} \to \End_F(A_f)^{\mathrm{op}} \otimes \Q\{G\}$.

\begin{thm}\label{thm EndAf}
Assume $f$ doesn't have CM. Then the map $\rho_f : \TT_{f,F} \to \End_F(A_f)^{\mathrm{op}} \otimes \Q\{G\}$ is bijective. In particular, every endomorphism of $A_f$ defined over $F$ arises from $\TT_{N,F}$.
\end{thm}

\begin{proof}
Since $\TT_{f,F}$ embeds in $\End_\Q(\Omega_{f,F})$, the map $\rho_f$ is injective. Let us prove that $\rho_f$ is surjective. Let $\mathcal{F}$ be the set of newforms $f^\sigma \otimes \chi$ with $\sigma : K_f \hookrightarrow \CC$ and $\chi \in \hat{G}$. For any $g \in \mathcal{F}$, let
\begin{equation}
\Omega_{f,F}[g] = (\Omega_{f,F} \otimes \Qb) \cap \Omega(\pi(g))
\end{equation}
denote the $g$-eigenspace of $\Omega_{f,F}$. By Proposition \ref{pro decomp Omegaf}, we have direct sum decompositions
\begin{align}
\label{decomp omegafF} \Omega_{f,F} \otimes \Qb & = \bigoplus_{g \in \mathcal{F}} \Omega_{f,F}[g],\\
\label{decomp omegafGg} \Omega_{f,F}[g] & = \bigoplus_{\substack{\sigma,\chi \\ f^\sigma \otimes \chi = g}} \omega_{f^\sigma} \cdot L_\chi.
\end{align}

By Lemma \ref{emb_equiv_F} and since $f$ doesn't have CM, we have $|\mathcal{F}| = |\Sigma| \cdot |\hat{G}| = [k:\Q] \cdot [F:\Q]$, and $\dim_{\Qb} \Omega_{f,F}[g] = [K_f:k]$ for every $g \in \mathcal{F}$, using notations from \S\ref{sec modular abvar}. By Lemma \ref{TK semisimple}, the map
\begin{equation}
\TT_{f,F} \otimes \Qb \to \prod_{g \in \mathcal{F}} \End_{\Qb} \Omega_{f,F}[g]
\end{equation}
is bijective. It follows that the rank of $\rho_f$ is
\begin{equation*}
\sum_{g \in \mathcal{F}} (\dim_{\Qb} \Omega_{f,F}[g])^2 = [k:\Q] \cdot [F:\Q] \cdot [K_f:k]^2
\end{equation*}
which agrees with the dimension of $\End_F(A_f) \otimes \Q\{G\}$ given by Theorem \ref{thm BfF}.
\end{proof}

\begin{remark}
It is well-known that $X_f = \End_{\Qb}(A_f) \otimes \Q$ is a crossed product algebra containing the Hecke field $K_f$ as a maximal commutative subalgebra \cite[Thm 5.1]{ribet:twists}. In fact, if $k_f$ denotes the center of $X_f$, then $X_f$ is a vector space of dimension $[K_f:k_f]$ over $K_f$, with an explicit $K_f$-basis of endomorphisms induced by the inner twists of $f$ \cite[\S 5]{ribet:twists}. It would be interesting to express these endomorphisms in terms of $\rho_f$.
\end{remark}

\section{Proofs of the main results}

Let us first recall Beilinson's theorem on modular curves \cite{beilinson:2}. Let $K$ be a compact open subgroup of $\GL_2(\A_f)$. For every $\pi \in \Pi(K)$, let $L(\pi,s)$ denote the Jacquet-Langlands $L$-function of $\pi$, with values in $\Qb \otimes \CC$, and shifted by $\frac12$ so that the functional equation corresponds to $s \leftrightarrow 2-s$. Note that $L(\pi,s)$ actually takes values in $E(\pi) \otimes \CC$, where $E(\pi) \subset \Qb$ is the number field generated by the values of the character $\theta_{\pi,K} : Z(\TT_K) \to \Qb$. If $f$ is a newform of weight $2$ with Fourier coefficients in $\Qb$ and $\pi(f)$ is the automorphic representation of $\GL_2(\A_f)$ associated to $f$, then we have $L(\pi(f),s)^\sigma = L(f^\sigma,s)$ for every embedding $\sigma : \Qb \hookrightarrow \CC$. The functional equation implies that the $L$-function $L(\pi,s)$ has a simple zero at each integer $m \leqslant 0$, with $L'(\pi,m) \in (E(\pi) \otimes \R)^\times$. Fix an integer $n \geqslant 2$. We have an isomorphism
\begin{equation}
H^2_{\mathcal{D}}(\overline{M}_K/\R,\R(n)) \cong H^1_B(\overline{M}_K(\CC),\R(n-1))^+.
\end{equation}
The Betti cohomology group decomposes with respect to the action of the Hecke algebra:
\begin{equation}
H^1_B(\overline{M}_K(\CC),\Q(n-1))^+ \otimes \Qb = \bigoplus_{\pi \in \Pi(K)} H(\pi)
\end{equation}
where $H(\pi)$ is the subspace cut out by the character $\theta_{\pi,K}$ acting on $\Omega(\pi)^K$.

Beilinson constructs a subspace $W_n \subset H^2_{\mathcal{M}}(\overline{M}_K,\Q(n))$ with the following property.

\begin{thm*}[Beilinson, \cite{beilinson:2} Thm 1.3]
Let $R=r_\BB(W_n) \subset H^1_B(\overline{M}_K(\CC),\R(n-1))^+$. We have a direct sum decomposition $R \otimes \Qb = \bigoplus_{\pi \in \Pi(K)} R(\pi)$ with $R(\pi)=L'(\pi,2-n) \cdot H(\pi)$ inside $H(\pi) \otimes \R$.
\end{thm*}

\begin{remark}
The localization sequence in $K$-theory implies that $H^2_{\mathcal{M}/\Z}(\overline{M}_K,\Q(n)) = \linebreak H^2_{\mathcal{M}}(\overline{M}_K,\Q(n))$ for any $n \geqslant 3$. In the case $n=2$, Schappacher and Scholl \cite[Thm 1.1.2(iii)]{schappacher-scholl} later proved that $W_2 \subset H^2_{\mathcal{M}/\Z}(\overline{M}_K,\Q(2))$. 
\end{remark}

Let us now reformulate Beilinson's theorem using the equivariant formalism of \S 1. The Hecke algebra $\TT_K$ acts on the Chow motive $H^1(\overline{M}_K)(n)$, thereby defining an element of $\CHM_{\Q}(\TT_K)$. The following result is probably well-known to the experts, but doesn't seem to appear in the literature.

\begin{thm}[Equivariant version of Beilinson's theorem]\label{beilinson equiv}
Conjecture \ref{conj4} holds for the equivariant $L$-value $L({}_{\TT_K} H^1(\overline{M}_K),n)$.
\end{thm}

\begin{proof}
By Proposition \ref{func1}, it suffices to prove that for some number field $E$, Conjecture \ref{conj4} holds for $L({}_A M,n)$ where $M=H^1(\overline{M}_K) \otimes_\Q E$ and $A=\TT_K \otimes_\Q E$. Let $E \subset \Qb$ be the number field generated by the (finitely many) fields $E(\pi)$ with $\pi \in \Pi(K)$. Note that for every such representation $\pi$, the space $\Omega(\pi)^K$ as a natural $E$-structure $\Omega(\pi)^K_E$, and that the character $\theta_{\pi,K}$ takes values in $E$. We have a direct sum decomposition $M = \bigoplus_{\pi \in \Pi(K)} M(\pi)$ in $\CHM_\Q(A)$, where the structural morphism $A \to \End(M(\pi))$ factors through $A_\pi := \End_E(\Omega(\pi)^K_E)$ (see Lemma \ref{TK semisimple}). Moreover $L({}_{A_\pi} M(\pi),s)=L(\pi,s)$ in $E \otimes \CC$. By Proposition \ref{func2}, it suffices to establish Conjecture \ref{conj4} for $L({}_{A_\pi} M(\pi),n)$.

By construction, the Beilinson subspace $W_n$ is stable under $\TT_K$. For any $\pi \in \Pi(K)$, let $W_n(\pi)$ be the subspace of $W_n \otimes_\Q E$ cut out by the character $\theta_{\pi,K}$. We may identify $W_n(\pi)$ with a subspace of $H^2_{\mathcal{M}/\Z}(M(\pi),E(n))$. Since the Beilinson regulator map is $\TT_K$-equivariant, we have $r_\BB(W_n(\pi)) \otimes_E \Qb =R(\pi)$. Note that $H(\pi)$ has a natural $E$-structure $H(\pi)_E$ and that $L'(\pi,2-n) \in (E \otimes \R)^\times$. By Beilinson's theorem, we have $R(\pi)=L'(\pi,2-n) \cdot H(\pi)$ and this implies $r_\BB(W_n(\pi)) = L'(\pi,2-n) \cdot H(\pi)_E$. This means precisely that the element $\vartheta_\infty(W_n(\pi))$ of $K_0(A_\pi,\R)$ is given by $\hat{\delta}(L'(\pi,2-n))$.
\end{proof}

\begin{thm}\label{main thm 1}
Let $f$ be a newform of weight $2$ without complex multiplication, and let $F$ be a finite abelian extension of $\Q$. Let $X=\End_F(A_f) \otimes\Q$ and $G=\Gal(F/\Q)$. For every integer $n \geqslant 2$, Conjecture \ref{conj4} holds for $L({}_{X\{G\}} A_f/F,n)$.
\end{thm}

\begin{proof}
Assume $f \in S_2(\Gamma_1(N))$ is a newform of level $N$. We use Theorem \ref{beilinson equiv} with the subgroup $K=K_1(N)_F$ defined in \ref{hecke base change}, so that $\overline{M}_K=X_1(N)_F$. Let $J_1(N)_F$ be the Jacobian of $X_1(N)_F$. We have an isomorphism $H^1(X_1(N)_F) \cong H^1(J_1(N)_F)$ in $\CHM_\Q(\TT_{N,F})$ (see for instance \cite[Prop 4.5]{scholl:motives} applied to $X=X_1(N)_F$ and $X'=J_1(N)_F$). Let $e_f \in \TT_{N,F}$ be the idempotent from Proposition \ref{pro ef}, and let $\TT_{f,F} =e_f \TT_{N,F} e_f$. By Theorem \ref{thm EndAf}, we have an isomorphism of Chow motives $e_f(H^1(J_1(N)_F)) = H^1(A_f/F)$ in $\CHM_\Q(X^{\mathrm{op}}\{G\})$. The result now follows from Proposition \ref{func3}.
\end{proof}

\begin{thm}\label{main thm 2}
Let $f$ be a newform of weight $2$ without complex multiplication, and let $F,F'$ be finite abelian extensions of $\Q$ such that $F \subset F'$. Let $X=\End_{F'}(B_{f,F}) \otimes \Q$ and $G=\Gal(F'/F)$. For every integer $n \geqslant 2$, Conjecture \ref{conj4} holds for $L({}_{X\{G\}} B_{f,F}/F',n)$.
\end{thm}

\begin{proof}
By definition of $B_{f,F}$, we have an isogeny $A_f \sim_F B_{f,F}^m$ for some $m \geqslant 1$, and thus an isomorphism of Chow motives $H^1(A_f/F') \cong H^1(B_{f,F}/F')^{\oplus m}$. Let $R=M_m(X\{G\}) \cong M_m(X) \{G\}$. Put $X' = \End_{F'}(A_f) \otimes \Q$ and $G'=\Gal(F'/\Q)$, so that we have a canonical embedding $R \hookrightarrow X'\{G'\}$. By Theorem \ref{main thm 1} and Proposition \ref{func2}, Conjecture \ref{conj4} holds for $L({}_{R^{\mathrm{op}}} H^1(B_{f,F}/F')^{\oplus m},n)$. We conclude by projecting onto $H^1(B_{f,F}/F')$ using Proposition \ref{func3}.
\end{proof}

Putting together Theorems \ref{strongly modular} and \ref{main thm 2}, we deduce the following result.

\begin{cor}\label{cor 1}
Let $A$ be an abelian variety over a number field $K$ such that $L(A/K,s)$ is a product of $L$-functions of newforms of weight $2$ without complex multiplication. Let $X=\End_K(A) \otimes \Q$. Then for every integer $n \geqslant 2$, Conjecture \ref{conj4} holds for $L({}_X A,n)$.
\end{cor}

In the particular case of $\Q$-curves, this gives the following result.

\begin{cor}\label{cor 2}
Let $E$ be a $\Q$-curve without complex multiplication over a number field $K$ such that $L(E/K,s)$ is a product of $L$-functions of newforms of weight $2$. Then for every integer $n \geqslant 2$, Conjecture \ref{conj4} holds for $L(E/K,n)$.
\end{cor}

This result has the following consequence on Zagier's conjecture on $L(E,2)$ (see \cite{brunault:LEF} for how to derive Corollary \ref{cor 3} from Corollary \ref{cor 2}).

\begin{cor}\label{cor 3}
Let $E$ be a $\Q$-curve without complex multiplication over a number field $K$ such that $L(E/K,s)$ is a product of $L$-functions of newforms of weight $2$. Then the weak form of Zagier's conjecture on $L(E/K,2)$ holds.
\end{cor}

We also get the following consequence on $L(E,3)$. Deninger predicted that for an elliptic curve $E/\Q$, the $L$-value $L(E,3)$ can be expressed in terms of certain double Eisenstein-Kronecker series evaluated at algebraic points of $E$ \cite{deninger:higher}. Goncharov proved this conjecture in \cite{goncharov:LE3} by explicitly computing the regulator map on $K_4(E)$ and applying Beilinson's theorem. Deninger's conjecture can be generalized to an elliptic curve over an arbitrary number field. Using Goncharov's techniques, we get the following result.

\begin{cor}\label{cor 4}
Let $E$ be a $\Q$-curve without complex multiplication over a number field $K$ such that $L(E/K,s)$ is a product of $L$-functions of newforms of weight $2$. Then the weak form of Deninger's conjecture on $L(E/K,3)$ holds.
\end{cor}

\bibliographystyle{plain}
\bibliography{../references}

\end{document}